\documentclass[10pt]{amsart}

\usepackage[hmargin=0.8in,height=8.6in]{geometry}
\usepackage{graphicx, amssymb, color}
\usepackage{amsmath}

\newtheorem{thm}{Theorem}[section]
\newtheorem{cor}[thm]{Corollary}
\newtheorem{lem}[thm]{Lemma}
\newtheorem{prop}[thm]{Proposition}
\theoremstyle{definition}
\newtheorem{defn}[thm]{Definition}

\newtheorem{ass}[thm]{Assumption}

\theoremstyle{remark}
\newtheorem{rem}[thm]{Remark}

\numberwithin{equation}{section}
\newcommand{\F}{\mathcal{F}}
\newcommand{\C}{\mathcal{C}}

\renewcommand{\P}{\mathbb{P}}
\newcommand{\E}{\mathbb{E}}

\newcommand{\R}{\mathbb{R}}

\renewcommand{\L}{\mathcal{L}}

\newcommand{\T}{\mathcal{T}}

\newcommand{\basis}{(\Omega,  \, (\F_t)_{t \in \R_+}, \, \P)}

\title[Regularity of the optimal stopping problems]{Regularity of the optimal stopping problem for jump diffusions}

\author{Erhan Bayraktar}
\address[Erhan Bayraktar]{Department of Mathematics, University of Michigan, Ann Arbor, MI 48109, USA; e-mail:erhan@umich.edu}
\thanks{The first author is supported in part by the National Science Foundation under an applied mathematics research grant and a Career grant, DMS-0906257 and DMS-0955463, respectively, and in part by the Susan M. Smith Professorship. The second author is supported in part by STICERD at London School of Economics}

\author{Hao Xing }
\address[Hao Xing]{Department of Statistics, London School of Economics, London, WC2A 2AE, UK; e-mail:h.xing@lse.ac.uk}
\keywords{Optimal stopping, variational inequality, L\'{e}vy processes, regularity of the value function, smooth fit principle, Sobolev spaces}

\begin{document}
\maketitle

\begin{abstract}
 The value function of an optimal stopping problem for jump diffusions is known to be a generalized solution of a variational
 inequality. Assuming that the diffusion component of the process is nondegenerate and a mild assumption on the singularity of the L\'{e}vy measure, this paper shows that the value function of this optimal stopping problem on an unbounded domain with finite/infinite variation jumps is in $W^{2,1}_{p, loc}$ with $p\in(1, \infty)$. As a consequence, the smooth-fit property holds.
\end{abstract}

\section{Introduction}\label{section:introduction}
On a probability space $\basis$, consider a one-dimensional jump diffusion process $X= \{X_t; t\geq 0\}$ whose dynamics is governed by the following stochastic differential equation:
\begin{equation}\label{eq:sde-x}
 d X_t = b(X_{t},t)\, dt + \sigma (X_{t},t)\, dW_t + \int_{\R} h(X_{t-}, y, t) \left(N(dt, dy) - 1_{\{|y|\leq 1\}} dt \,\nu(dy)\right),
\end{equation}
in which $W= \{W_t; t\geq 0\}$ is a $1$-dimensional Wiener process, $N$,
independent of the Wiener process, is a Poisson random measure on $\R_+ \times \R$ with its mean measure $dt \times \nu(dy)$, and $\nu$ is a L\'{e}vy measure on $\R$. The coexistence of diffusions and infinite activity jumps is motivated by recent studies of A\"{i}t-Sahalia and Jacod in \cite{ait-jacod} and \cite{ait-jacod-Brownian}.

This paper studies the problem of maximizing the discounted terminal
reward $g$ by optimally stopping the process $X$ before a fixed
time horizon $T$. The value function of this problem is defined as
\begin{equation}\label{eq:optimal-stopping}\tag{OS}
 u(x,t) = \sup_{\tau \in \T_{t,T}} \E^{t,x}\left[e^{-(\tau-t)r(X_t)} g(X_{\tau})\right],
\end{equation}
in which $\T_{t,T}$ is the set of all stopping times valued between $t$ and $T$. A specific example of such an optimal stopping problem is the American option pricing problem, where $X$ models the logarithm of the stock price process and $g$ represents the pay-off function.

The value function $u$ is expected to satisfy a variational inequality with a nonlocal integral term (see e.g. Chapter 3 of \cite{Bensoussan-Lions}).
Different concepts of solutions were employed to characterize the value function: Pham used the notion of viscosity solution in \cite{pham-control}. Also see \cite{MR2330346}, \cite{MR2422079} for more recent results in this direction. Lamberton and Mikou worked with L\'{e}vy processes and showed in \cite{lamberton} that the value function can be understood in the distributional sense.

When the diffusion component in $X$ is nondegenerate, the value function is expected to have higher degree of regularity. Sections 1-3 in Chapter 3 of \cite{Bensoussan-Lions} and \cite{garr-mena} analyzed the Cauchy problems for second order partial integro-differential equations and showed the existence and uniqueness of solutions in both Sobolev and H\"{o}lder
spaces. Also see \cite{MR680451}. The intuition is that the diffusions component dominates the contribution from jumps in determining the regularity of solutions, no matter whether jumps have finite variation or not. However this intuition is only a folklore theorem for obstacle problems. There are some limited results available whose assumptions on obstacles, domains, and the structure of the jumps may not be appropriate for financial applications. For example, Bensoussan and Lions analyzed an obstacle problem for jump diffusions where jumps may have finite/infinite activity with finite/infinite variation; see in Theorem 3.2 in \cite{Bensoussan-Lions} on pp. 234. However, their assumption on the obstacle may not be satisfied by option payoffs. In the mathematical finance literature, when irregular obstacles are considered, the jumps are usually restricted to finite activity or infinite activity with finite variation cases. Zhang studied in \cite{zhang} an obstacle problem for a jump diffusion with finite active jumps. Also see \cite{pham}, \cite{yang}, \cite{bayraktar-finite-horizon}, and \cite{bayraktar-xing-SIMA} for further developments. More recently, Davis et al. in \cite{guoliang10}, generalizing the results in \cite{MR2486085} for the diffusion case, analyzed an impulse control problems for jump diffusions with infinite activity but finite variation jumps. A regularity result which treats obstacle problems with irregular obstacles and infinite variation jumps has been missing in the literature.

In this paper, we allow for infinite activity, infinite variation jumps. We show in Theorem~\ref{thm: main} that the value function of an obstacle problem solves a variational inequality for almost all points in the domain, and that it is an element in $W^{2,1}_{p, loc}$  with $p\in(1, \infty)$ (see later this section for the definition of this Sobolev space). This regularity result directly implies that the smooth fit property holds and the value function is $C^{2,1}$ inside the continuation region. These results confirm the intuition that the nondegenerate diffusions components dominate any type of L\'{e}vy jumps in determining the regularity of the value function for obstacle problems. We also develop a non-local version of the interior Schauder estimate in Proposition~\ref{prop:W^21_p-est}, which could be useful to study other integro-differential equations with irregular initial conditions.

The remainder of the paper is organized as follows. After introducing notation at the end of this section, main results are presented in Section~\ref{sec:main results}. Regularity properties of the infinitesimal generator of $X$ are analyzed in Section~\ref{sec: reg prop}. Then main results are proved in Section~\ref{sec:proofs}.

\subsection{Notation}\label{sec:notation}
For a given open interval $D= (\ell, r)$ with $-\infty \leq \ell < r \leq\infty$, let us define the $\delta$-neighborhood of $D$ as $D^\delta := (\ell-\delta, r+\delta)$ for $\delta>0$. We will also denote $D_s := D \times (0,s)$, $D^\delta_s := D^\delta \times (0,s)$ for any $s>0$, $E_s:= \R \times [0, s]$, and by $\overline{A}$ the closure of the indicated set $A$. Let us recall definitions of Sobolev spaces and H\"{o}lder spaces in what follows; see \cite{lad} pp. 5-7 for further details.

\begin{defn}\label{def:sobolev_holder space}
$C^{2,1}(D_s)$ denotes the class of continuous functions on $D_s$ with continuous
classical time and spatial derivatives up to the first and second order respectively.

For any positive integer $p\geq 1$, $W^{2,1}_{p}(D_s)$ is the space of functions $v \in L_p(D_s)$ with generalized derivatives $\partial_t v$,
$\partial_{x} v$, $\partial^2_{xx} v$, and a finite norm
$\|v\|_{W^{2,1}_p(D_s)}:= \|\partial_t v\|_{L_p(D_s)} + \|\partial_{x} v\|_{L_p(D_s)} + \|\partial^2_{xx} v\|_{L_p(D_s)}$.
The space $W^{2,1}_{p, \, loc}(D_s)$ consists of functions whose $W^{2,1}_p$-norm is finite on any compact subsets of $D_s$.

For any positive nonintegral real number $\alpha$,
$H^{\alpha,\alpha/2}\left(\overline{D_s}\right)$ is the
space of functions $v$ that are continuous in $\overline{D_s}$ with continuous classical derivatives
$\partial_t^r\partial_x^s v$ for $2r+s<\alpha$, and have finite norm
$\|v\|^{(\alpha)}_{\overline{D_s}} := |v|_x^{(\alpha)} + |v|_t^{(\alpha/2)} + \sum_{2r+s\leq [\alpha]} \|\partial_t^r\partial_x^s
 v\|^{(0)}$, in which $\|v\|^{(0)} = max_{D_s} |v|$, $|v|_x^{(\alpha)} = \sum_{2r+s=[\alpha]} \sup_{|x-x'|\leq \rho_0}
\frac{|\partial_t^r\partial_x^s v(x,t)-\partial_t^r\partial_x^s v(x',t)|}{|x-x'|^{\alpha-[\alpha]}}$, and $|v|_t^{(\alpha/2)} = \sum_{\alpha-2<2r+s<\alpha} \sup_{|t-t'|\leq \rho_0}
\frac{|\partial_t^r\partial_x^s v(x,t)-\partial_t^r\partial_x^s v(x,t')|}{|t-t'|^{(\alpha-2r-s)/2}}$, for a constant $\rho_0$.
The space $H^{\alpha}\left(\overline{\Omega}\right)$ is the H\"{o}lder space when only the spatial variable is considered.
\end{defn}

\section{Main results}\label{sec:main results}
\subsection{Model}
Let us first specify the jump diffusion $X$ in \eqref{eq:sde-x}.
We assume that the drift and the volatility of $X$, the discounting factor $r$, and the jump size $h$ satisfy the following set of assumptions:
\begin{ass}\label{ass: coeff}
Let $a:= \frac12 \sigma^2$. Coefficients $a, b, r \in H^{\ell, \frac{\ell}{2}}(E_T)$ for some $\ell>1$, $r(x,t)\geq 0$. Moreover, there exist a strictly positive constant $\lambda$ such that $a(x,t)\geq \lambda$  for all $(x,t)\in E_T$. The jump size $h(x, y, t)$ is continuously differentiable in $x$ and $\partial_x h(x,y,t)$ is H\"{o}lder continuous in $(x,t)$, moreover there exists a constant $C$ such that
\begin{equation}\label{ass: h}
|h(x,y,t)| \leq C|y| \quad \text{ and } |h(x_1, y, t_1) - h(x_2, y, t_2)| \leq C|y|\,(|x_1-x_2| + |t_1-t_2|^{\frac12}),
\end{equation}
for $(x,t), (x_i, t_i)\in \R \times [0,T]$, $i=1$ or $2$, and $y\in \R$.
\end{ass}

Without loss of generality, we will take $C$ in \eqref{ass: h} to be equal to one, otherwise we would rescale the process $X$.
For the pure jump component in \eqref{eq:sde-x}, we assume that $\nu$ is a L\'{e}vy measure on $\R$. See \cite{sato} for this terminology. In particular, we require that $\int_{\R} (y^2 \wedge 1) \,\nu(dy) <\infty$. When $h\equiv y$, the jump component of \eqref{eq:sde-x} is a L\'{e}vy process.
The aforementioned assumptions on coefficients and the jump component ensure that \eqref{eq:sde-x} admits a unique strong solution (see \cite{Gihman-Skorohod}), which we denote by $X$. This jump diffusion process $X$ is said to have \emph{finite activity}, if $\nu$ is a finite measure on $\R$, otherwise it is said to have \emph{infinite activity}. We say  that the jumps of $X$ have \emph{finite variation}, if $\int_\R |y| \nu(dy) <\infty$, otherwise we say that they have \emph{infinite variation}.

Among all possible L\'{e}vy measures, we consider the following large subclass in this paper:
\begin{ass}\label{ass: Levy measure}
 The L\'{e}vy measure satisfies $\int_{|y|>1} |y|^2 \nu(dy)<\infty$. Moreover it has a density, which we denote by $\rho$, and this density satisfies $\rho(y)\leq \frac{M}{|y|^{1+\alpha}}$ on $|y|\leq 1$, for some constants $M>0$ and $\alpha\in [0,2)$.
\end{ass}
Note that the interval $|y|\leq 1$ can be replaced by any other neighborhood of $0$ in our analysis, our choice of this interval is made for notational convenience.
\begin{rem}\label{remark:mean-measure-sing}
 Virtually all L\'{e}vy processes used in the financial modeling satisfy above assumption.
For jump diffusions models, $\nu$ is a finite measure as in Merton's and Kou's model. For normal tempered stable processes, $\rho$ has a power singularity $1/|y|^{1+2\beta}$ at $y=0$, with $0\leq \beta <1$; see (4.25) in \cite{cont}. In particular, this class contains Variance Gamma and Normal Inverse Gaussian where $\beta = 0$ or $1/2$ respectively. For generalized tempered stable processes (see Remark 4.1 in \cite{cont}),
$
 \rho(y) = \frac{C_-}{|y|^{1+\alpha_-}} e^{-\lambda_- |y|} 1_{\{y<0\}} + \frac{C_+}{|y|^{1+\alpha_+}} e^{-\lambda_+ y} 1_{\{y>0\}},
$
with $\alpha_-, \alpha_+ <2$ and $\lambda_-, \lambda_+ >0$. In particular, CGMY processes in \cite{cgmy} and regular L\'{e}vy processes of exponential type (RLPE) in \cite{boy-leven} are special examples of this class.
\end{rem}

Having introduced the jump diffusion process $X$, let us discuss the problem \eqref{eq:optimal-stopping}. We assume that the payoff function $g$ satisfies the following set of assumptions:
\begin{ass}\label{ass: payoff}
The payoff function $g$ is positive, bounded and Lipschitz continuous on $\R$. That is, there exists positive constants $K$ and $L$ such that $0\leq g(x)\leq K$ for any $x\in \R$ and $|g(x)-g(y)|\leq L|x-y|$ for any $x,y\in \R$. Moreover $g$ satisfies $\partial^2_{xx}g \geq -J$ for some positive constant $J$ in the distributional sense, i.e., $\int_{\R} g(x) \partial^2_{xx} \phi(x) \, dx \geq -J \int_{\R} \phi(x) \, dx$ for any compactly supported smooth function $\phi$ on $\R$.
\end{ass}
A typical example, where these assumptions holds, is the American put option payoff $g(x) = (K-e^x)_+$ for some $K\in \R_+$.

For the problem \eqref{eq:optimal-stopping}, we define its
continuation region $\mathcal{C}$ and stopping region
$\mathcal{D}$ as usual:
\begin{equation*}
 \mathcal{C} := \left\{(x,t)\in \R^n\times [0,T) : u(x,t) >
 g(x)\right\} \quad \text{ and } \quad \mathcal{D} :=
 \left\{(x,t)\in \R^n\times [0,T) : u(x,t) = g(x) \right\}.
\end{equation*}

\subsection{Main regularity results}
Intuitively, one can expect from It\^{o}'s formula that the value function $u$
satisfies the following \emph{variational inequality}:
\begin{equation}\label{eq:varineq-u}
\begin{split}
 &\min\left\{(-\partial_t  - \L + r)\, u, u -g \right\} = 0, \quad (x,t)\in \R \times [0,T),\\
 &u(x,T) = g(x), \hspace{3.3cm} x\in \R.
\end{split}
\end{equation}
Here, the integro-differential operator $\L$ is the infinitesimal
generator of $X$. Its application on a smooth test function $\phi$ is
\begin{equation}
  \L\phi := \L_D \phi + I \phi,   \label{eq:def-L}
\end{equation}
where $\L_D \phi(x,t) := a(x,t) \,\partial^2_{xx} + b(x,t)\,\partial_x$ and the integral term
\begin{equation}\label{eq:def-I}
 I \phi(x,t) := \int_{\R} \left[\phi(x+h(x,y,t),t) - \phi(x,t) - h(x,y,t) \, \partial_x \phi(x,t) \, 1_{\{|y|\leq 1\}} \right]\nu(dy).
\end{equation}
In what follows we will not write down the arguments of $h$ explicitly or only indicate the argument that we are focusing in order to keep the notation simple.

In general, one does not know a priori whether $u$ is sufficiently smooth so that it solves \eqref{eq:varineq-u} in the classical sense. Moreover, it is not even clear whether $Iu$ is well defined in the classical sense. When $\phi(\cdot, t)$ is Lipschitz continuous on $\R$ with a Lipschitz continuous derivative $\partial_x \phi(\cdot, t)$ in a neighborhood of $x$, it can be shown that $I\phi(x,t)$ is well defined in the classical sense. Indeed,
$
 I\phi(x,t) = I_{\epsilon} \phi(x,t) + I^{\epsilon} \phi(x,t)<\infty,
$ where
\begin{eqnarray}
 I^{\epsilon} \phi(x,t) &:=& \int_{|y|>\epsilon}\left[\phi(x+h,t) -
 \phi(x,t)\right]\nu(dy) - \partial_x \phi(x,t)  \int_{\epsilon < |y| \leq 1} h\, \nu(dy) \nonumber \\
 & \leq & C \int_{|y|>\epsilon} |y| \,\nu(dy) + |\partial_x \phi(x,t)| \int_{\epsilon < |y| \leq 1} |y|\,
 \nu(dy), \nonumber\\
 I_{\epsilon} \phi(x,t) &:=& \int_{|y|\leq \epsilon} \left[\phi(x+h,t) - \phi(x,t) - h \,\partial_x \phi(x,t)\right]
 \nu(dy) \label{eq:Isubeps}\\
 &=& \int_{|y|\leq \epsilon} h  \left(\partial_{x} \phi(z,t) - \partial_{x}
 \phi(x,t)\right) \nu(dy)
 \leq  C \int_{|y|\leq \epsilon}  \, y^2 \nu(dy) \nonumber.
\end{eqnarray}
Here, the first inequality follows from the Lipschitz continuity of  $\phi(\cdot, t)$ and the assumption that $|h|\leq |y|$; the mean value theorem implies the second equality in \eqref{eq:Isubeps} where $z$ satisfies $|z-x|<|h|$; the last inequality holds due to the Lipschitz continuity of  $\partial_x \phi(\cdot, t)$ in an $\epsilon$-neighborhood of $x$. However, the value function $u$, in general, does not have these regularity properties mentioned above. We only know from Lemma~3.1 in \cite{pham-control} that $u$ is Lipschitz continuous in $x$ and $1/2-$H\"{o}lder continuous in $t$.
Nevertheless, we will see that the integral term $Iu$ is well defined in the classical sense in Lemma~\ref{lemma:infinite-var-holder} below. In fact, more is true as we show in the next theorem, which is the main result of the paper.

\begin{thm}\label{thm: main}
 Let Assumptions~\ref{ass: coeff}, \ref{ass: Levy measure}, and  \ref{ass: payoff} hold. Then $u\in W^{2,1}_{p, loc}(\R \times (0,T))$ for any integer $p\in (1, \infty)$. Moreover, $u$ solves \eqref{eq:varineq-u} for almost every point in $E_T$.
\end{thm}

The following corollary is of special interest for the American option problem.

\begin{cor}\label{cor: smooth-fit}
 Under the assumptions of Theorem~\ref{thm: main},
 \begin{enumerate}
  \item[(i)] $\partial_x u \in C(\R\times [0,T))$, i.e., the smooth-fit holds;
  \item[(ii)] $u \in C^{2,1}$ in the region where $u>g$.
 \end{enumerate}
\end{cor}

\begin{rem}\label{rem: finite var}
 When jumps of $X$ have finite variation, i.e., $\int_{\R} |y| \wedge 1 \, \nu(dy)<\infty$, the proof of the  main result is much simpler. This is because, when jumps of $X$ have finite variation, the infinitesimal generator $\L$ can be rewritten so that its integral component has a reduced form. For any test function $\phi$ that is Lipschitz continuous in its first variable, $\L \phi$ can be decomposed as $\L\phi = \L_D^f \phi+ I^f \phi$, in which $\L_D^f \phi =  a \,\partial^2_{xx} \phi + [b - \int_{|y|\leq 1} h \,\nu(dy)]\,\partial_{x} \phi$ and
\begin{equation}\label{eq:def-If}
  I^{f} \phi(x, t) := \int_{\R} \left[\phi(x+h, t) - \phi(x,t)\right] \, \nu(dy).
\end{equation}
The previous integral is clearly well defined. Indeed $|I^f \phi(x,t)| \leq C \int_{\R} |y|\, \nu(dy) <+\infty$ follows from the Lipschitz continuity of $\phi(\cdot, t)$ and $|h|\leq |y|$. Moreover, $I^f \phi$ is also H\"{o}lder continuous in its both variables; see Lemma~\ref{lemma:fin-var-int} below. Since the value function $u$ is known to be Lipschitz continuous in its first variable (see Lemma~3.1 in \cite{pham-control}), $I^f u$ is already well defined and H\"{o}lder continuous. Therefore, in order to study the regularity of $u$, $I^f u$ can be treated as a driving term in \eqref{eq:varineq-u}. However, this simplification cannot be applied when jumps of $X$ have infinite variation, i.e., $\int_{\R} (|y|\wedge 1) \nu(dy) =\infty$.
\end{rem}

\section{Regularity properties of the integro-differential operator}\label{sec: reg prop}
\subsection{The integral operator}
The integral operator $I$ has two basic features. First, $\nu$ has a singularity at $y=0$. As a result, $I$ maps functions with certain degree of regularity to functions with less regularity. This is contrast to the case in which $\nu$ is a finite measure. In that case $\int_\R \phi(x+h, t) \nu(dy)$ is already well defined, for any $\phi$ with at most linear growth, and this integral has the same regularity as $\phi$; see \cite{yang}. Second, $I$ is a nonlocal operator. Therefore, regularity of $I\phi$ on a given interval $D$ depends on $\phi$ outside $D$. In this subsection, we shall study these two features in detail and analyze the regularity of $I\phi$ when $\phi$ is either a function in certain H\"{o}lder or Sobolev spaces.

Consider $I$ as an operator between H\"{o}lder spaces. When jumps of $X$ have finite variation, we can work with the reduced integral operator $I^f$ in \eqref{eq:def-If}. It has the following regularity property.

\begin{lem}\label{lemma:fin-var-int}
 Let Assumption~\ref{ass: Levy measure} hold with $0\leq \alpha<1$ and $s>0$. For any $\phi$ which is Lipschitz continuous in its first variable and $1/2$-H\"{o}lder continuous in its second variable,
 \[
 \begin{array}{ll}
  I^f \phi \in H^{1-\gamma, \frac{1-\gamma}{2}}(\overline{D_s}) \quad  \forall \gamma \in (0,1), & \text{ when } \alpha =0;\\
  I^f \phi \in H^{1-\alpha, \frac{1-\alpha}{2}}(\overline{D_s}), & \text{ when } 0<\alpha<1.
 \end{array}
 \]
\end{lem}

However when jumps of $X$ have infinite variation,  the integral term $I^f \phi$ is no longer well defined for Lipschitz continuous functions. Hence we work with $\L$ and its integral part $I$ in the forms of \eqref{eq:def-L} and \eqref{eq:def-I}. We will see that if we choose an appropriate test function $\phi$, $I\phi$ is still well defined and H\"{o}lder continuous in both its variables. Regularity estimates of the following type have been obtained in \cite{silvestre} and \cite{Mikulevicius-Zhang}.

\begin{lem}\label{lemma:infinite-var-holder}
 Let Assumption~\ref{ass: Levy measure} hold with $\alpha\in [1,2)$ and $s>0$.
 \begin{enumerate}
 \item[(i)] Suppose that $\phi$ satisfies $|\phi(x_1, t_1) - \phi(x_2, t_2)| \leq L(|x_1-x_2|+ |t_1-t_2|^{\frac12})$ for some $L>0$ and any $(x_1, t_1), (x_2, t_2)\in E_s$. If, moreover,  $\phi \in H^{\beta ,
 \frac{\beta}{2}}(\overline{D_s^1})$ for some $\beta\in (\alpha,
 2)$, then $I \phi\in H^{\beta-\alpha-\gamma,
 \frac{\beta-\alpha-\gamma}{2}}\left(\overline{D_s}\right)$ and
 \begin{equation}\label{eq:Iu-holder-norm}
  \left\|I\phi\right\|^{\left(\beta-\alpha-\gamma\right)}_{\overline{D_s}} \leq
  C \left(L + \|\phi\|^{(\beta)}_{\overline{D_s^1}}\right),
 \end{equation}
 for a positive constant $C$ depending on $D$, $\alpha$, and $\beta$.

 \item[(ii)] If $\phi\in H^{\beta, \frac{\beta}{2}}(E_s)$ for some $\beta\in (\alpha,
 2)$, then $I\phi \in H^{\beta-\alpha-\gamma,
 \frac{\beta-\alpha-\gamma}{2}}(E_s)$ and
 \begin{equation}\label{eq:Iu-holder-norm-R}
  \left\|I
  \phi\right\|_{E_s}^{\left(\beta-\alpha-\gamma\right)}
  \leq C \, \|\phi\|^{(\beta)}_{E_s},
 \end{equation}
 for a positive constant $C$ depending on $\alpha$ and $\beta$.
 \end{enumerate}
 Here $\gamma=0$ when $\alpha \in (1,2)$; $\gamma$ is an arbitrary number in $(0, \beta -\alpha)$ when $\alpha =1$.
\end{lem}

Since the proofs of Lemmas~\ref{lemma:fin-var-int} and \ref{lemma:infinite-var-holder} are similar, we only present the proof of Lemma~\ref{lemma:infinite-var-holder}.

\begin{proof}[{\bf Proof of Lemma~\ref{lemma:infinite-var-holder}}]
Statement (ii) is a special case of Statement (i) when the domain is taken to be $\R$, instead of $D$. In particular, $\|\cdot \|^{(\beta)}_{E_s} \geq L$; see Definition~\ref{def:sobolev_holder space}. It then suffices to prove statement (i).  For notational simplicity, $C$ represents a generic constant throughout the rest of proof.

 \textit{\underline{Step 1}: Estimate $\max_{\overline{D_s}} |I\phi|$.}
 For any $(x,t)\in \overline{D_s}$,
 \begin{eqnarray*}
  \left| I\phi(x,t) \right| &\leq& \int_{|y|\leq 1} \left|\phi(x+h, t) - \phi(x,t) -
  h \, \partial_{x} \phi(x,t)\right| \nu(dy) + \int_{|y|>1} \left|\phi(x+h, t) -
  \phi(x,t)\right|\nu(dy)\\
  &\leq& \int_{|y|\leq 1} |h|\,\left|\partial_{x} \phi(z, t) - \partial_{x} \phi(x, t)\right| \nu(dy) + L \int_{|y|>1} |h| \,\nu(dy) \nonumber\\
  &\leq& \|\phi\|^{(\beta)}_{\overline{D_s^1}} \int_{|y|\leq 1} |y|^{\beta} \nu(dy) + L \int_{|y|>1}
  |y| \nu(dy) \nonumber\\
  &\leq& C\left(L +\|\phi\|^{(\beta)}_{\overline{D_s^1}}\right), \nonumber
 \end{eqnarray*}
 where the second inequality follows from the mean value theorem with $|z-x|\leq |h|\leq |y|\leq 1$; the third inequality is the result of the $(\beta-1)$-H\"{o}lder continuity of $\partial_x \phi$ on $\overline{D_s^1}$ and $|h|\leq |y|$; the fourth inequality holds thanks to Assumption~\ref{ass: Levy measure}.

 \textit{\underline{Step 2}: Show that $I\phi$ is H\"{o}lder continuous in $x$.} For $x_1, x_2\in D$ and $t\in [0,s]$, we break up $|I\phi(x_1, t)-I\phi(x_2, t)|$ into three parts:
 \begin{eqnarray*}
  && \left|I\phi(x_1, t) - I\phi(x_2, t)\right| \leq I_1 + I_2 +I_3, \quad
  \text{ in which } \label{eq:infinite-var-I}\\
  && I_1(x,t) := \int_{|y|\leq \epsilon} \left[\left|\phi(x_1+h(x_1), t) - \phi(x_1, t) - h(x_1) \, \partial_x \phi(x_1,
  t)\right| \right.\\
  && \hspace{2.3cm}+  \left.\left|\phi(x_2+h(x_2), t) - \phi(x_2, t) - h(x_2) \,\partial_x \phi(x_2, t)\right|\right] \nu(dy),\nonumber\\
  && I_2(x,t) := \int_{\epsilon <|y| \leq 1}\left[\left|\phi(x_1+h(x_1), t) - \phi(x_1, t) -  \phi(x_2+h(x_2), t) + \phi(x_2, t)\right| \right.\\
  && \hspace{2.7cm}+ \left.\left|h(x_1)\partial_x \phi(x_1, t) - h(x_2)\partial_x \phi(x_2,
  t)\right|\right] \nu(dy), \nonumber\\
  && I_3(x,t) := \int_{|y|>1} \left[\left|\phi(x_1+h(x_1), t) - \phi(x_2+h(x_2), t)\right| + \left|\phi(x_1,
  t)-\phi(x_2,t)\right|\right] \nu(dy), \nonumber
 \end{eqnarray*}
 where variables $y$ and $t$ are ignored in $h$ and the constant $\epsilon \leq 1$ will be determined later. Let us estimate each above integral term separately. First, an estimate similar to that in Step 1 shows that
 $
  I_1 \leq 2 \,\|\phi\|^{(\beta)}_{\overline{D_s^1}} \int_{|y|\leq \epsilon} |y|^{\beta}
  \nu(dy)  =C \|\phi\|^{(\beta)}_{\overline{D_s^1}}\,
  \epsilon^{\beta-\alpha}.
 $
 Second, it follows from the Lipschitz continuity of $\phi(\cdot, t)$, $(\beta-1)$-H\"{o}lder continuity of $\partial_x \phi$, and $|h|\leq |y|$ that
 \begin{eqnarray*}
  &&\left|\phi(x_1+h(x_1), t) - \phi(x_1, t) -  \phi(x_2+h(x_2), t) + \phi(x_2, t)\right|\\
  &\leq& \left|\phi(x_1+h(x_1), t) - \phi(x_1, t) -  \phi(x_2+h(x_1), t) + \phi(x_2, t)\right| + |\phi(x_2+h(x_1), t) - \phi(x_2 + h(x_2), t)|\\
  &\leq&\left|\int_0^{h(x_1)} \left|\partial_x \phi(x_1 + z, t) - \partial_x \phi(x_2 + z, t)\right| dz\right| + C|x_1-x_2|\,|y|\\
  &\leq& \|\phi\|^{(\beta)}_{\overline{D^1_s}} \, |x_1-x_2|^{\beta-1} |y| + C|x_1-x_2|\,|y|.
 \end{eqnarray*}
 Similarly, $\left|h(x_1)\partial_x \phi(x_1, t) - h(x_2)\partial_x \phi(x_2,
  t)\right| \leq \|\phi\|^{(\beta)}_{\overline{D^1_s}} \,|y| \,(|x_1-x_2| + |x_1-x_2|^{\beta-1})$. Therefore,
 \begin{eqnarray*}
  I_2 &\leq& \int_{\epsilon<|y|\leq 1} C\left(1+ \|\phi\|^{(\beta)}_{\overline{D^1_s}}\right) |x_1-x_2|^{\beta-1} |y|\nu(dy) \leq C \left(1+ \|\phi\|^{(\beta)}_{\overline{D_s^1}}\right) \ |x_1-x_2|^{\beta-1} \cdot \left\{\begin{array}{ll} \epsilon^{1-\alpha} -1 & \text{ when } 1<\alpha<2\\
   -\log{\epsilon} & \text{ when } \alpha=1
  \end{array}\right.,
 \end{eqnarray*}
 where the second inequality follows from Assumption~\ref{ass: Levy measure}. Third, it is clear from the Lipschitz continuity of $\phi$ and \eqref{ass: h} that $I_3 \leq C |x_1-x_2| \int_{|y|>1} (1+|y|) \nu(dy)$.

 Now pick $\epsilon = |x_1-x_2|\wedge 1$. Since $1\leq \alpha
 <2$ and $\beta>\alpha$, we have $\epsilon^{\beta-\alpha}\leq
 |x_1-x_2|^{\beta-\alpha}$, $\epsilon^{-\alpha}-1 \leq
 |x_1-x_2|^{-\alpha}$, $\epsilon^{1-\alpha}-1\leq
 |x_1-x_2|^{1-\alpha}$ and $-\log{\epsilon} \leq \frac{1}{\alpha+\gamma-1}
 |x_1-x_2|^{1-\alpha-\gamma}$. All above estimates combined imply that
 \begin{equation*}\label{eq:inf-var-holder-x}
  |I\phi(x_1,t) - I\phi(x_2,t)| \leq C \left(1 +
  \|\phi\|^{(\beta)}_{\overline{D_s^1}}\right)|x_1-x_2|^{\beta-\alpha-\gamma},
 \end{equation*}
 for a constant $C$ independent of $x_1$, $x_2$, and $t$.

 \textit{\underline{Step 3}: Show that $I\phi$ is H\"{o}lder continuous in $t$.} The proof is similar to that in Step 2. First we separate $|I\phi(x, t_1) - I\phi(x, t_2)|$ into three parts as above. Then using  $|\partial_x \phi(x, t_1) -\partial_x \phi(x, t_2)| \leq \|\phi\|^{(\beta)}_{\overline{D_s^1}} \,|t_1-t_2|^{\frac{\beta-1}{2}}$ (see Definition~\ref{def:sobolev_holder space}), together with semi-H\"{o}lder continuity of $h(x, \cdot)$ in \eqref{ass: h}, and choosing $\epsilon=|t_1-t_2|^{\frac12}\wedge 1$, we can obtain
 \begin{equation*}\label{eq:inf-var-holder-t}
 |I\phi(x,t_1) - I\phi(x, t_2)| \leq C \left(1 + \|\phi\|^{(\beta)}_{\overline{D_s^1}}\right) |t_1-t_2|^{\frac{\beta-\alpha-\gamma}{2}},
\end{equation*}
for a constant $C$ independent of $x$, $t_1$, and $t_2$.
\end{proof}

When $I$ is considered as an operator between Sobolev spaces, it maps $W^{2,1}_p-$functions to $L_p-$functions on a smaller domain.

\begin{lem}\label{lemma:L_p-Iv}
Let Assumption~\ref{ass: Levy measure} hold. Consider a function $\phi\in W^{2,1}_{p}(D\times (t_1, t_2))$ such that $\phi$ is bounded and $\partial_x \phi$ is locally bounded on $\R\times [t_1, t_2]$. Then for any $\eta>0$ and $\alpha\in[0,2)$,
\begin{equation}\label{eq:Lp-est-Iv}
 \|I \phi \|_{L_p(D \times (t_1,t_2))} \leq C \eta^{2-\alpha}
 \|\phi\|_{W^{2,1}_p(D^{\eta} \times (t_1, t_2))} +
 C \left(\max_{\R\times [t_1, t_2]} |\phi| + \max_{D^{1}\times [t_1,t_2]}
 |\partial_x \phi|\right) \cdot \left\{\begin{array}{ll} (1+ \eta^{1-\alpha}),\, \alpha \neq 1 \\
 (1-\log{\eta}),\, \alpha=1\end{array}\right.,
\end{equation}
for some constant $C$ depending on $D$, $t_1$, and $t_2$.
\end{lem}

\begin{rem}\label{remark:L-p est finite var}
 When $X$ has finite variation jumps, i.e., $0\leq \alpha<1$, $\eta$ in \eqref{eq:Lp-est-Iv} can be chosen as zero. Hence $L_p-$norm of $I\phi$ only depends on $\max_{\R\times [t_1, t_2]} |\phi|$ and $\max_{D^{1}\times [t_1,t_2]} |\partial_x \phi|$.
\end{rem}

\begin{proof}
 Utilizing truncation and smooth mollification, one can construct a sequence of smooth function $(\phi^\epsilon)_{\epsilon>0}$ such that $\phi^\epsilon$ converges to $\phi$ in $W^{2,1}_{p}$ and $I\phi^\epsilon$ converges to $I\phi$ in $L_p$ as $\epsilon \rightarrow 0$ (c.f. Section 5.3 and Appendix C.4 in \cite{evans}). Therefore, it suffices to prove the statement for a smooth function $\phi$.

 To this end, observing that $\phi(x+h,t) - \phi(x,t) - h \partial_x \phi(x,t) = h^2 \int_0^1
  (1-z) \,\partial^2_{xx} \phi(x+zh, t) \, dz$,
 the integral $I\phi$ can be bounded above by three integral terms:
  \begin{equation*}
 \begin{split}
  \left|I \phi(x,t)\right|
  \leq& \int_{|y|\leq \eta} h^2\, \nu(dy)\int_0^1
  dz \left|\partial^2_{xx} \phi(x+zh, t)\right|\\
  & + \int_{\eta< |y|\leq 1} \nu(dy) \left|\phi(x+h,t) - \phi(x,t) - h \partial_x \phi(x,t)
  \right| + \int_{|y|>1} \nu(dy) \left|\phi(x+h, t) - \phi(x,t)\right|\\
  =:&  I_1+ I_2 + I_3.
 \end{split}
 \end{equation*}
 In the rest of proof, the $L_p$-norm of each above term is estimated respectively. First,
 \begin{equation*}\label{eq:Lp-I_ij}
 \begin{split}
 \left\|I_1(\cdot, t)\right\|^p_{L_p(D)} & = \int_D dx \left[\int_{|y|\leq \eta} h^2 \nu(dy) \int_0^1 dz \left|\partial^2_{xx}\phi(x+zh,
 t)\right|\right]^p \\
 & \leq \int_D dx \int_0^1 dz \left[\int_{|y|\leq \eta} \nu(dy) \, h^2 \left|\partial^2_{xx}\phi(x+zh,
 t)\right|\right]^p\\
 & \leq C \int_D dx \int_0^1 dz \left[\int_{|y|\leq \eta} dy \,|y|^{1-\alpha} \left|\partial^2_{xx} \phi(x+zh,
 t)\right|\right]^p\\
 & \leq C \int_D dx \int_0^1 dz \left(\int_{|y|\leq \eta} dy\,
 |y|^{1-\alpha}\right)^{\frac{p}{q}} \cdot \int_{|y|\leq \eta} dy\, |y|^{1-\alpha} \left|\partial^2_{xx}\phi(x+zh,
 t)\right|^p\\
 &\leq C \left(\int_{|y|\leq \eta} dy\,
 |y|^{1-\alpha}\right)^{\frac{p}{q}} \cdot \int_{|y|\leq \eta} dy\, |y|^{1-\alpha} \int_0^1 dz\int_D dx \left|\partial^2_{xx}\phi(x+zh,
 t)\right|^p\\
 &\leq  C \,\eta^{(2-\alpha)p} \left\|\partial^2_{xx} \phi(\cdot,
 t)\right\|^p_{L_p(D^\eta)},
 \end{split}
 \end{equation*}
 where the first inequality follows from Fubini's theorem and
Jensen's inequality since $p>1$; the second inequality is a result of the assumption that $|h|\leq |y|$ and Assumption~\ref{ass: Levy measure}; the third inequality follows from
 H\"{o}lder inequality with $1/p+1/q=1$; the fourth inequality utilizes Fubini's theorem; and the fifth inequality
 holds since $x+zh\in D^\eta$ for any $|h|\leq |y|\leq \eta$ and $z\in [0,1]$.
 Second, since $x+h\in D^1$ for $x\in D$, and $|h|\leq |y|\leq 1$, it follows that
 \begin{equation*}\label{eq:Lp-I_2}
 \begin{split}
  \left\|I_2 (\cdot, t)\right\|_{L_p(D)} \leq C \max_{D^1 \times [t_1, t_2]}  \left|\partial_x \phi\right| \cdot \int_{\eta\leq |y|\leq 1} |y| \nu(dy)\leq C \max_{D^1 \times [t_1,t_2]} \left|\partial_x \phi\right| \cdot \left\{\begin{array}{ll} (1+ \eta^{1-\alpha}),\, \alpha \neq 1 \\
 (1-\log{\eta}),\, \alpha=1\end{array}\right.
 \end{split}
 \end{equation*}
 Third, it is clear that
 $
  \left\|I_3 \phi(\cdot, t)\right\|_{L_p(D)} \leq C \cdot
 \max_{\R \times [t_1,t_2]} |\phi|
 $,
 since $\phi$ is bounded.

 Now, recall $\|I\phi\|_{L_p(D\times(t_1,t_2))} := \left[\int_{t_1}^{t_2} \|I\phi(\cdot, t)\|_{L_p(D)} \,
 dt\right]^{\frac1p}$. The statement then follows from above $L_p$-norm estimates on $I_k$, $k=1,2,3$.
\end{proof}

\subsection{An interior estimate}
The $L_p-$norm estimate of the integral term in Lemma~\ref{lemma:L_p-Iv} helps to derive the following $W^{2,1}_p-$norm estimate for solutions of the Cauchy problem below. This estimate is a nonlocal version of the parabolic Calderon-Zygmund estimate (c.f. Theorem~9.1 in \cite{lad} pp.341).

\begin{prop}\label{prop:W^21_p-est}
Suppose that Assumptions~\ref{ass: coeff} and \ref{ass: Levy measure} are satisfied. Let $v$ be a $W^{2,1}_{p, loc}-$solution of the following Cauchy problem:
\begin{equation*}\label{eq:pide-cauchy}
\begin{split}
 &\left(\partial_t - \L_{D} - I + r\right) v = f(x,t), \hspace{1cm} (x,t)\in \R\times (0,T],\\
 & v(x,0) = g(x), \hspace{3.4cm} x\in \R,
\end{split}
\end{equation*}
where $f\in L_{p, loc}(E_T)$. If $v$ is bounded and $\partial_x v$ is locally bounded on $E_T$, then for any $s\in (0,T)$, there exist $\delta\in (0,s)$ and $C_{\delta}$, depending on $\delta$, such that
\begin{equation}\label{eq:v-W^21_p-est}
 \left\| v \right\|_{W^{2,1}_p(D\times (s,T))} \leq
 C_{\delta} \left[\max_{E_T} |v| + \max_{D^{\delta/4 +1}\times [0,T]} |\partial_x v| + \| f\|_{L^p(D^{\delta/4} \times (\frac{\delta}{2}, T))}\right].
\end{equation}
\end{prop}

\begin{rem}
 The main idea of the following proof is to treat $Iv$ as a driving term and utilize the classical Calderon-Zygmund estimate for local PDEs. However, as we have seen in Lemma~\ref{lemma:L_p-Iv}, $W^{2,1}_p-$norm of $v$ controls $L_p-$norm of $Iv$, which in turn bounds the $W^{2,1}_p-$norm of $v$ via the Calderon-Zygmund estimate. Therefore, a careful balance between extending domains and controlling $W^{2,1}_p-$norm of $v$ needs to be maintained in the following proof. This is contrast to the case where only finite variation jumps are considered. As we have seen in Remark~\ref{remark:L-p est finite var}, $\max |\partial_x v|$ and $\max |v|$ control the $L_p-$norm of $Iv$ which bounds the $W^{2,1}_p$-norm of $v$. Hence, in the case of finite variation jumps, \eqref{eq:v-W^21_p-est} can be obtained directly from the classical Calderon-Zygmund estimate for local PDEs.
\end{rem}

\begin{proof}
 The constant $C$ denotes a generic constant throughout this proof. Domains used in this proof are displayed in Figure~1.

 For a constant $\delta\in (0,s)$ which will be determined later, let us choose a cut-off function $\zeta^{\delta}$ such that $0\leq \zeta^\delta\leq 1$, $\zeta^\delta = 1$ inside $D\times (\delta, T)$ and $\zeta^\delta =0$ outside $D^{\delta/4} \times (\delta/2, T)$.
 Moreover $\zeta^\delta$ can be chosen to satisfy
 \begin{equation}\label{eq:prop-cut-off}
  \left|\partial_{x} \zeta^{\delta}\right| \leq
  \frac{C}{\delta}, \quad \left|\partial^2_{xx}
  \zeta^{\delta}\right|\leq \frac{C}{\delta^2}, \quad \text{and}
  \quad \left|\partial_t \zeta^{\delta}\right|\leq
  \frac{C}{\delta}.
 \end{equation}
 \begin{figure}[b]\label{figure:diagrams}
\begin{minipage}{\textwidth}
\begin{center}
\caption{Domains used in this proof}
\includegraphics[width=4in, height=2.5in]{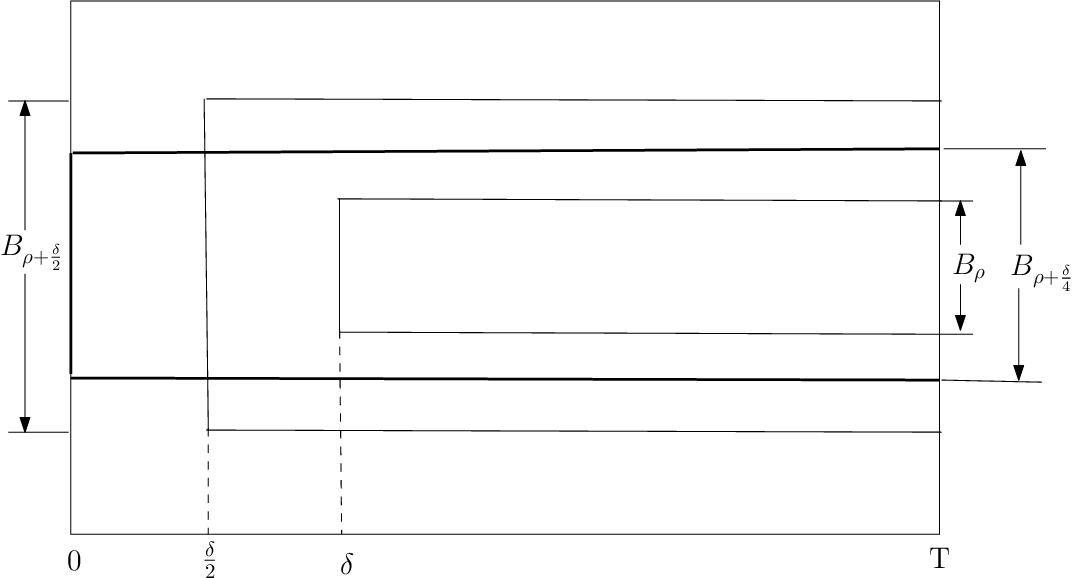}
\end{center}
\end{minipage}
\end{figure}
The function $w: = \zeta^{\delta}\, v$ satisfies
 \begin{equation*}
 \begin{split}
  & \left(\partial_t - \L_D + r\right)\, w = \zeta^{\delta}
 \, Iv(x,t) + \zeta^{\delta}\, f(x,t) + h(x,t), \quad
  (x,t)\in D^{\delta/4}\times (0,T),\\
  & w(x,t) =0, \hspace{6.4cm} (x,t)\in \partial \overline{D^{\delta/4}}
  \times [0,T),\\
  & w(x,0) = 0, \hspace{6.4cm} x\in \overline{D^{\delta/4}},
 \end{split}
 \end{equation*}
 in which $h:= \partial_t \zeta^{\delta} \, v - a \left(\partial^2_{xx} \zeta^{\delta} \, v + 2\, \partial_{x} \zeta^{\delta} \, \partial_{x} v\right) - b \, \partial_{x} \zeta^{\delta} \, v$. Appealing to
 Theorem~9.1 in \cite{lad} pp.341, we can find a constant $C$
 such that
 \begin{equation}\label{eq:w^21_p-est-1}
 \begin{split}
  \left\|w\right\|_{W^{2,1}_p (D^{\delta/4}\times
  (0,T))} \leq & C \left[\left\|\zeta^{\delta} \, Iv\right\|_{L_p} + \left\|\zeta^{\delta} \, f\right\|_{L_p} + \left\|h\right\|_{L_p}\right],
 \end{split}
 \end{equation}
 where all $L_p$-norms on the right-hand-side are taken on $D^{\delta/4}\times
  (0,T)$.

  In what follows, we will estimate the terms on the right-hand-side
  of \eqref{eq:w^21_p-est-1} respectively. First, when $\alpha\neq 1$,
  \begin{equation*}\label{eq:w^21_p-est-2}
  \begin{split}
   \left\|\zeta^{\delta} \, Iv\right\|_{L_p(D^{\delta/4}\times
  (0,T))} & \leq \left\|Iv\right\|_{L_p(D^{\delta/4}\times
  (\frac{\delta}{2},T))}\\
  & \leq C \left(\frac{\delta}{4}\right)^{2-\alpha}
  \left\|v\right\|_{W^{2,1}_p(D^{\delta/2}\times
  (\frac{\delta}{2},T))} + C \left(1+
  \left(\frac{\delta}{4}\right)^{1-\alpha}\right) \left[\max_{E_T} |v| + \max_{D^{\delta/4 +1}\times [0,T]} \left|\partial_x
  v\right|\right],
  \end{split}
  \end{equation*}
  where the first inequality follows from the choice of $\zeta^\delta$; the second inequality follows from Lemma~\ref{lemma:L_p-Iv} with $\eta=\delta/4$, $t_1=\delta/2$, and $t_2=T$. When $\alpha=1$, a similar estimate can be obtained. In that case, the rest of proof is similar to that for $\alpha\neq 1$ case, hence we only present the proof for $\alpha\neq 1$ henceforth.
  Second, it is clear that
  $
   \left\|\zeta^{\delta} \, f\right\|_{L_p(D^{\delta/4}\times
   (0,T))} \leq \left\|f\right\|_{L_p(D^{\delta/4}\times
   (\frac{\delta}{2},T))}.
  $
  Third, we will estimate the $L_p-$norm of $h$. To this end, let us derive a bound for $\|\partial_t \zeta^\delta v\|_{L_p(D^\delta \times (0,T))}$ in what follows. It follows from \eqref{eq:prop-cut-off} that
  \begin{equation*}\label{eq:w^21_p-est-4}
  \begin{split}
   \left\|\partial_t \zeta^{\delta} \, v\right\|_{L_p(D^\delta \times
   (0,T))} \leq C \max_{E_T} |v| \, \delta^{-1} Area\left(D^{\delta/4}\times(\delta/2, T)\setminus D \times (\delta, T)\right)^{\frac1p} \leq C \max_{E_T} |v| \, \delta^{\frac{1-p}{p}},
   \end{split}
  \end{equation*}
  where $Area(\cdot)$ is the Lebesgue measure. Estimates on other terms of $h$ can be performed similarly to obtain
  $$
   \|h\|_{L_p(D^{\delta/4}\times (0,T))} \leq C \left(\delta^{\frac{1-p}{p}} + \delta^{\frac{1-2p}{p}}\right)\left(\max_{E_T} |v| + \max_{D^{\delta/4}\times [0,T]} |\partial_x v|\right).
  $$
  Utilizing above estimates on the right-hand-side of \eqref{eq:w^21_p-est-1}, we obtain
  \begin{equation*}\label{eq:w^21_p-est-8}
  \begin{split}
   \left\|v\right\|_{W^{2,1}_p(D\times (\delta, T))}  \leq &
   \left\|w\right\|_{W^{2,1}_p(D^{\delta/4}\times (0,
   T))}\\
    \leq & C \left(\frac{\delta}{4}\right)^{2-\alpha}
  \left\|v\right\|_{W^{2,1}_p(D^{\delta/2}\times
  (\frac{\delta}{2},T))} + C \left(1+ \delta^{1-\alpha} + \delta^{\frac{1-p}{p}} +
  \delta^{\frac{1-2p}{p}}\right) \, \left(\max_{E_T} |v| + \max_{D^{\delta/4 +1}\times [0,T]} \left|\partial_x
  v\right|\right)\\
  & + \left\|f\right\|_{L_p(D^{\delta/4}\times (\frac{\delta}{2},T))}.
  \end{split}
  \end{equation*}
 Multiplying $\delta^2$ on both hand sides of the previous inequality,
  \begin{equation*}\label{eq:w^21_p-est-9}
  \delta^2  \left\|v\right\|_{W^{2,1}_p(D \times (\delta,
  T))} \leq 4 C \left(\frac{\delta}{4}\right)^{2-\alpha} \,
  \left(\frac{\delta}{2}\right)^2 \left\|v\right\|_{W^{2,1}_p(D^{\delta/2}\times (\frac{\delta}{2},
  T))} + K(\delta),
  \end{equation*}
  where
  $
   K(\delta) = C \left(\delta^2+ \delta^{3-\alpha} + \delta^{\frac{1+p}{p}} +
  \delta^{\frac{1}{p}}\right) \left(\max_{E_T} |v| + \max_{D^{\delta/4 +1} \times [0,T]} \left|\partial_x
  v\right|\right)+ \delta^2 \left\|f\right\|_{L_p(D^{\delta/4}\times
   (\frac{\delta}{2},T))}
  $. Denote $F(\tau) := \tau^2 \left\|v\right\|_{W^{2,1}_p(D^{\delta-\tau}\times (\tau,
  T))}$. The previous inequality gives the following recursive inequality
  \begin{equation*}
   F(\delta) \leq 4 C \left(\frac{\delta}{4}\right)^{2-\alpha}
   F\left(\frac{\delta}{2}\right) + K(\delta).
  \end{equation*}
  Now choosing a sufficiently small $\delta \in (0,s)$ such that $4 C \left(\delta/4\right)^{2-\alpha} \leq \frac12$, we obtain from the above inequality that
  \begin{equation*}
   F(\delta) \leq \frac12 \,F\left(\frac{\delta}{2}\right) + K(\delta).
  \end{equation*}
   Note that $F(\delta)$ is finite for any $\delta$, since the $W^{2,1}_p-$norm of $v$ is finite in any compact domain of $\R\times (0,T)$, and $K(\delta)$ is increasing in $\delta$. We then obtain from iterating the previous inequality that
  \[
   F(\delta) \leq \sum_{i=0}^{\infty} \frac{1}{2^i}\,
   K\left(\frac{\delta}{2^i}\right) \leq \sum_{i=0}^{\infty}
   \frac{1}{2^i}\, K(\delta) = 2\, K(\delta).
  \]
  In terms of $W^{2,1}_{p, loc}-$norms, the previous inequality reads
  \begin{equation*}
  \begin{split}
   \left\|v\right\|_{W^{2,1}_p (D\times (s,T))} & \leq 2 \, C \left[1+ \delta^{1-\alpha} + \delta^{\frac{1-p}{p}} +
  \delta^{\frac{1-2p}{p}}\right] \, \left[\max_{E_T} |v| + \max_{D^{\delta/4 +1}\times [0,T]} \left|\partial_x
  v\right|\right]+ 2\, \left\|f\right\|_{L_p(D^{\delta/4}\times
   (\frac{\delta}{2},T))}\\
   &\leq C_{\delta} \left[\max_{E_T} |v| + \max_{D^{\delta/4+1}\times [0,T]} \left|\partial_x
  v\right| + \left\|f\right\|_{L_p(D^{\delta/4}\times (\frac{\delta}{2},T))}\right].
  \end{split}
  \end{equation*}
\end{proof}

\section{Proof of main results}\label{sec:proofs}
\subsection{The penalty method} We use the penalty method (see e.g. \cite{friedman-freeboundary} and \cite{yang}) to analyze the following variational inequality:
\begin{equation}\label{eq: VI}
\begin{split}
 &\min\left\{\left(\partial_t - \L_{D} - I + r\right) v,
 v-g\right\} =0, \hspace{1cm} (x,t)\in \R\times (0,T], \\
 & v(x,0) = g(x), \hspace{4.6cm} x\in \R.
\end{split}
\end{equation}
The nonlocal integral term introduces several technical difficulties in applying the penalty method.
In this section, we will focus on the case where $X$ has infinite variation jumps, i.e., Assumption~\ref{ass: Levy measure} holds with $1\leq \alpha<2$. When $X$ has finite variation jumps, i.e., $0\leq \alpha<1$, the integral operator has the reduced form $I^f$ in \eqref{eq:def-If}, see Remark \ref{rem: finite var}. Then all proofs are similar but easier than those in the infinite variation case.

For each $\epsilon\in(0,1)$, consider the following penalty problem:
\begin{equation}\label{eq:penalty-v-eps}
\begin{split}
 & \left(\partial_t - \L_{D} - I + r\right) v^{\epsilon} +
 p_{\epsilon} \left(v^{\epsilon} - g^{\epsilon}\right) = 0, \hspace{1cm}
 (x,t)\in \R \times (0,T],\\
 & v^{\epsilon}(x,0) = g^{\epsilon}(x), \hspace{4.6cm} x\in \R,
\end{split}
\end{equation}
Here $\left\{g^{\epsilon}\right\}_{\epsilon\in(0,1)}$ is a mollified sequence of $g$ such that $\partial^2_{xx} g^\epsilon(x) \geq -J$, $0\leq g\leq K$, and $|(g^\epsilon)'(x)|\leq L$ for any $x\in \R$; see \cite{friedman-freeboundary} pp.27 for its construction. The mollified sequence can be chosen such that constants $J, K$, and $L$, appearing in Assumption~\ref{ass: payoff}, are independent of $\epsilon$. The penalty term $p_{\epsilon}(y)\in
 C^{\infty}(\R)$ is chosen to satisfy
 following properties:
 \begin{equation}\label{eq:prop-p-eps}
 \begin{split}
 &(i)\, p_{\epsilon}(y) \leq 0, \quad  (ii)\, p_{\epsilon}(y)=0
 \text{  for } y\geq \epsilon, \quad (iii)\, p_{\epsilon}(0) = - a^{(0)} J - |b|^{(0)} L - r^{(0)} K - J \int_{|y|\leq 1} |y|^2 \nu(dy) - K \int_{|y|>1} \nu(dy),\\
 &(iv)\, p^{'}_{\epsilon}(y) \geq 0, \quad   (v)\, p^{''}_{\epsilon}(y)
 \leq 0, \quad \text{ and } \quad (vi) \, \lim_{\epsilon\downarrow 0}
 p_{\epsilon}(y) = \left\{\begin{array}{ll}0, & y>0\\ -\infty, &
 y<0\end{array}\right.,
 \end{split}
 \end{equation}
 where $a^{(0)}= \max_{E_T}a$, $|b|^{(0)}=\max_{E_T}|b|$, and $r^{(0)} = \max_{E_T}r$ are finite thanks to Assumption~\ref{ass: coeff}. Indeed, $p_\epsilon$ can be chosen as a smooth mollification of the function $\min\{-2 p_{\epsilon}(0) x/\epsilon + p_\epsilon(0), 0\}$.

 Now we show that each penalty problem \eqref{eq:penalty-v-eps} has a classical solution. To this end, let us first recall  the Schauder fixed point theorem (see Theorem 2 in \cite{friedman} pp. 189).
 \begin{lem}\label{lemma:schauder-fixed-pt}
  Let $\Theta$ be a closed convex subset of a Banach space and let
  $\T$ be a continuous operator on $\Theta$ such that $\T\Theta$ is
  contained in $\Theta$ and $\T\Theta$ is precompact. Then $\T$ has
  a fixed point in $\Theta$.
 \end{lem}

 \begin{lem}\label{lemma:existence-penalty-prob}
  Let Assumptions~\ref{ass: coeff}, \ref{ass: Levy measure} with $1\leq \alpha<2$, and  \ref{ass: payoff} hold. Then for any $\epsilon\in (0,1)$ and $\beta\in (\alpha, 2)$, \eqref{eq:penalty-v-eps} has a solution $v^{\epsilon}\in H^{2+\beta-\alpha-\gamma, 1+
  \frac{\beta-\alpha-\gamma}{2}}(E_T)$. Here $\gamma=0$ when $1<\alpha<2$; $\gamma$ is an arbitrary number in $(0, \beta-\alpha)$ when $\alpha =1$.
 \end{lem}
 \begin{proof}
  We will first prove the existence on a sufficiently small time interval
  $[0,s]$ via the Schauder fixed point theorem, then extend this solution to the interval $[0,T]$.

  Let us consider the set $\Theta := \left\{v \in H^{\beta, \frac{\beta}{2}}(E_s) \text{ with its H\"{o}lder norm } \|v\|^{(\beta)}_{E_s} \leq
  U_0\right\}$, where $s$ and $U_0$ will be
  determined later. It is clear that $\Theta$ is a bounded, closed
  and convex set in the Banach space $H^{\beta,
  \frac{\beta}{2}}(E_s)$. For any $v\in \Theta$, consider the
  following Cauchy problem for $u-g^{\epsilon}$:
  \begin{equation}\label{eq:def-T}
  \begin{split}
   &\left(\partial_t - \L_{D} + r\right) (u - g^{\epsilon}) =
   Iv - p_{\epsilon} (v-g^{\epsilon}) + (\L_{D} - r)\,
   g^{\epsilon},  \quad (x,t)\in \R\times (0,s],\\
   &u(x,0) - g^{\epsilon}(x) = 0, \hspace{5.9cm} x\in \R.
  \end{split}
  \end{equation}
  We define an operator $\T$ via $u = \T v$ using the solution $u$ of \eqref{eq:def-T}. Let us check the conditions for the
  Schauder fixed point theorem are satisfied in the following four steps:

  \textit{\underline{Step 1}: $Tv$ is well defined.}
  Since $v\in H^{\beta, \frac{\beta}{2}}(E_s)$ with $\beta\in (\alpha,
  2)$, Lemma~\ref{lemma:infinite-var-holder} part (ii) implies that
  $Iv\in H^{\beta-\alpha-\gamma, \frac{\beta-\alpha-\gamma}{2}}(E_s)$
  with
  $
   \left\|Iv\right\|^{(\beta-\alpha-\gamma)}_{E_s} \leq C \,
   \|v\|^{(\beta)}_{E_s}
  $.
  On the other hand, using properties of $v$, $g^\epsilon$ and $p_\epsilon$, one can check that $-p_\epsilon (v-g^\epsilon) + (\L_D-r) g^\epsilon \in H^{\beta-\alpha-\gamma, \frac{\beta-\alpha-\gamma}{2}}(E_s)$. Therefore, Theorem 5.1 in \cite{lad} pp. 320 implies that
  \eqref{eq:def-T} has a unique solution $u-g^{\epsilon} \in H^{2+\beta-\alpha-\gamma, 1+
  \frac{\beta-\alpha-\gamma}{2}}(E_s)$. Hence $u = Tv \in H^{2+\beta-\alpha-\gamma, 1+
  \frac{\beta-\alpha-\gamma}{2}}(E_s)$, since $g^\epsilon$ is smooth.

  \textit{\underline{Step 2}. $\T\Theta \subset \Theta$.} It follows from
  Lemma 2 in \cite{friedman} pp. 193 that there exists
  a positive constant $A_{\beta}$, depending on $\beta$, such that
  \begin{equation}\label{eq:beta-norm-est}
  \begin{split}
   \left\|u-g^{\epsilon}\right\|^{(\beta)}_{E_s} & \leq A_{\beta} s^{\xi}
   \left[\|Iv\|^{(0)}_{E_s} + \|p_{\epsilon}(v-g^{\epsilon})\|^{(0)}_{E_s} + \left\|(\L_{D}-r)\,
   g^{\epsilon}\right\|^{(0)}_{E_s}\right] \\
   & \leq A_{\beta} C s^{\xi} \|v\|^{(\beta)}_{E_s} + \widetilde{A},
  \end{split}
  \end{equation}
  where $\xi = \frac{2-\beta}{2}$, $C$ is the constant in Step 1, and $\widetilde{A}$ is a
  sufficiently large constant. Let $s$ be such that $\tau := A_{\beta} C s^{\xi}
  <1/2$ and let $U_0 := \max\{ \frac{2\widetilde{A}}{1-2\tau},
  2\,\|g^{\epsilon}\|^{(\beta)}_{E_s}\}$. Since $\|v\|^{(\beta)}_{E_s}\leq
  U_0$, it then follows from \eqref{eq:beta-norm-est} that
  \begin{equation}\label{eq:u_leq_U0}
   \|u\|^{(\beta)}_{E_s} \leq \|u-g^{\epsilon}\|^{(\beta)}_{E_s} +
   \|g^{\epsilon}\|^{(\beta)}_{E_s} \leq \tau U_0 + \widetilde{A} +
   \frac{U_0}{2} \leq \tau\, U_0 + \frac{1-2\tau}{2}\, U_0 + \frac{U_0}{2} = U_0.
  \end{equation}
  This confirms that $u=\T v\in \Theta$.

  \textit{\underline{Step 3}. $\T \Theta$ is a precompact subset of $H^{\beta,
  \frac{\beta}{2}}(E_s)$.} For any $\eta\in (\beta, 2)$, an estimate similar
  to \eqref{eq:beta-norm-est} shows that for any $v\in \Theta$,  $\|Tv\|^{(\eta)}_{E_s}
  \leq U_1$ for some constant $U_1$ depending on $U_0$ and $s$. Since bounded subsets of $H^{\eta,
  \frac{\eta}{2}}(E_s)$ are precompact subsets of
  $H^{\beta,\frac{\beta}{2}}(E_s)$ (see Theorem 1 in
  \cite{friedman} pp.188), then $\T \Theta$ is a
  precompact subset in $H^{\beta,\frac{\beta}{2}}(E_s)$.

  \textit{\underline{Step 4}. $\T$ is a continuous operator.} Let $v_n$ be a sequence
  in $\Theta$ such that $\lim_{n\rightarrow \infty} \|v_n-v\|^{(\beta)}_{E_s} = 0$, we will
  show $\lim_{n\rightarrow \infty} \|Tv_n-Tv\|^{(\beta)}_{E_s} = 0$.
  From \eqref{eq:def-T}, $w\triangleq Tv_n - Tv$ satisfies the
  Cauchy problem
  \begin{equation*}
  \begin{split}
   & \left(\partial_t - \L_{D} + r\right) w  = I (v_n-v)
    - \left[p_{\epsilon}(v_n - g^{\epsilon}) -
   p_{\epsilon}(v-g^{\epsilon})\right], \quad (x,t)\in \R
   \times(0,s],\\
   & w(x,0) = 0, \hspace{7.4cm} x\in \R.
  \end{split}
  \end{equation*}
  It follows again from Lemma~2 in \cite{friedman} pp. 193 that
  \begin{equation*}
  \begin{split}
   \|\T v_n - \T v\|^{(\beta)}_{E_s} = \|w\|^{(\beta)}_{E_s} &\leq A_{\beta} s^{\gamma}
   \left[\|I(v_n-v)\|^{(0)}_{E_s} + \left\|p_{\epsilon}(v_n-g^{\epsilon}) - p_{\epsilon}(v-g^{\epsilon})\right\|
   ^{(0)}_{E_s}\right]\\
   & \leq A_{\beta} s^{\gamma} \left[ C \|v_n-v\|^{(\beta)}_{E_s} +
   \max_{E_s, n} \left|p_{\epsilon}^{'} (v_n-g^{\epsilon})\right|
   \|v_n - v\|^{(0)}_{E_s}\right] \rightarrow 0 \quad \text{ as }
   n\rightarrow \infty.
  \end{split}
  \end{equation*}

  Now all conditions of the Schauder fixed point theorem are checked, hence
  $\T$ has a fixed point in $H^{\beta, \frac{\beta}{2}}(E_s)$, which is denoted
  by $v^{\epsilon}$. Moreover, it follows from results in
  Step 1 that $v^{\epsilon} = \T v^{\epsilon} \in H^{2+\beta-\alpha-\gamma, 1+
  \frac{\beta-\alpha-\gamma}{2}}(E_s)$.

  Finally, let us extend $v^{\epsilon}$ to the interval $[0,T]$. We can replace
  $g^{\epsilon}(\cdot)$ by $v^{\epsilon}(\cdot, s)$ in \eqref{eq:def-T}, since $\|v^{\epsilon}(\cdot,
  s)\|_{\R}^{(2+\beta-\alpha-\gamma)}$ is finite thanks to the result after Step 4 and because the choice of $s$ in Step 2 only depends on $\beta$ and
  $C$. If we choose a sufficiently large $U_0$, depending on $\|v^{\epsilon}(\cdot,
  s)\|_{\R}^{(2+\beta-\alpha-\gamma)}$, such that
  \eqref{eq:u_leq_U0} holds on $[s, 2s]$, then $\|v^{\epsilon}(\cdot,
  2s)\|_{\R}^{(2+\beta-\alpha-\gamma)}$ is finite thanks to the argument after Step 4. Now one can repeat this procedure to extend the time interval by $s$ each
  time, until it contains $[0,T]$.
 \end{proof}

 After the existence of classical solutions for \eqref{eq:penalty-v-eps} is established, we will study properties of the sequence $(v^\epsilon)_{\epsilon\in(0,1)}$ in the rest of this subsection. The following maximum principle is a handy tool in our analysis.

 \begin{lem}\label{lemma:maximum-principle}
 Suppose that $a>0$, $a$ and $b$ are bounded and the Levy measure $\nu$ satisfies $\int_{|y|>1} |y| \nu(dy) <\infty$. Assume also that we are given a function $c$ bounded from below on $E_T$. If $v\in C^0(E_T) \cap C^{2,1}(E_T)$ satisfies $\left(\partial_t - \L_{D} - I + c\right) v(x,t) \geq 0$ and $v$ is bounded from below on $E_T$, then $v(x,0)\geq 0$ for $x\in \R$ implies that $v\geq 0$ on $E_T$.
\end{lem}

\begin{proof}
 Let $v\geq -m$ and $c \geq -C_0$ on $E_T$ for some positive constants $m$ and $C_0$. For any positive $R_0$, consider the following function:
 \[
  w(x,t) := \frac{m}{f(R_0)} \left(f(|x|) + C_1 t\right) e^{C_0 t}, \quad (x,t)\in E_T,
 \]
 where $C_1$ will be determined later and $f: \R_+ \rightarrow \R_+$ is an increasing $C^2$ function such that $f=0$ in a neighborhood of $0$ and $f(R) =\frac{R^2}{1+ R}$ for sufficiently large $R$. It is clear that $\lim_{R\rightarrow +\infty} f(R) =\infty$ and derivatives $f^{'}$ and $f^{''}$ are bounded. Then $If(|x|)$ is bounded on $\R$. Indeed, there exists a constant $C$ such that
 \[
 \begin{split}
  \big| I f(\left|x\right|)\big| &  \leq \int_{|y|\leq 1} \nu(dy)\, \int_0^1 dz \, (1-z)
  \,h^2 \left|\partial^2_{xx} f(\left|x+z
  h\right|)\right| + \int_{|y|>1} \nu(dy) \left|f(\left|x+h\right|) -
  f(\left|x\right|)\right| \\
  & \leq C\left(\int_{|y|\leq 1} y^2 \nu(dy) + \int_{|y|>1} |y|\,
  \nu(dy)\right) <+\infty.
 \end{split}
 \]
 Combining above estimate with $c+C_0 \geq 0$, one can find a sufficient large constant $C_1$ such that
 \[
  (\partial_t - \L_D - I + c) w = e^{C_0 t} \frac{m}{f(R_0)} \left[C_1 + (c+C_0)(f(|x|)+ C_1 t) - a \,\partial^2_{xx} f(|x|) - b\, \partial_x f(|x|) - If(|x|)\right]>C_0 m, \quad \text{ on } E_T.
 \]
Now define $\tilde{v}:= v+w$. The previous estimate gives
 \begin{equation}\label{eq: est geq 0}
  (\partial_t - \L_D - I + c+C_0) \tilde{v} > C_0 v + C_0 m \geq 0, \quad \text{ for any } (x,t)\in E_T.
 \end{equation}
 On the other hand, $\tilde{v}(x,0) = \frac{m}{f(R_0)} f(|x|) + v(x,0) \geq 0$ due to $v(x,0) \geq 0$, moreover $\tilde{v}(x,t) \geq m + v(x,t)\geq 0$ for $|x| \geq R_0$ because $f$ is increasing and $v\geq -m$. Therefore, we claim that $\tilde{v}\geq 0$ for $(x,t)\in [-R_0, R_0] \times [0, T]$. Indeed, if there exists $(x,t)\in [-R_0, R_0] \times (0,T]$ such that $\tilde{v}(x,t)<0$, $\tilde{v}$ must take its negative minimum at some point $(x_0,t_0) \in [-R_0, R_0]\times (0, T]$. Note that this is also a global minimum for $\tilde{v}$ on $E_T$, hence $I\tilde{v}(x_0, t_0) \geq 0$, $\partial_t \tilde{v}(x_0, t_0)\leq 0$, $\partial_x \tilde{v}(x_0, t_0)=0$, $\partial^2_{xx} \tilde{v}(x_0, t_0)\geq 0$, and $(c+C_0) \tilde{v}(x_0, t_0)\leq 0$. As a result,
 $
  (\partial_t - \L_D - I +c+C_0) \tilde{v}(x_0, t_0) \leq 0,
 $
 which contradicts with \eqref{eq: est geq 0}. Now for fixed point $(x,t)$, the statement follows from sending the constant $R_0$ in $\tilde{v}$  to $\infty$.
\end{proof}
This maximum principle implies the uniqueness of classical solutions for the penalty problem \eqref{eq:penalty-v-eps}.
\begin{cor}\label{cor: uniqueness}
 Under assumptions of Lemma~\ref{lemma:existence-penalty-prob}, $v^{\epsilon}$ is the unique bounded classical solution of \eqref{eq:penalty-v-eps}.
\end{cor}
\begin{proof}
 Lemma~\ref{lemma:existence-penalty-prob} and the definition of H\"{o}lder spaces combined imply that $v_1:=v^\epsilon$ is a bounded classical solution. Now suppose there exists another solution $v_2$, then $v_1-v_2$ satisfies
 \begin{equation*}
 \begin{split}
  &\left(\partial_t - \L_D - I + r\right)\, (v_1-v_2) +
  p_{\epsilon}(v_1-g^{\epsilon}) - p_{\epsilon} (v_2-g^{\epsilon})
  =0, \quad (x,t)\in \R \times (0,T],\\
  & (v_1-v_2) (x,0) =0, \hspace{6.7cm} x\in \R.
 \end{split}
 \end{equation*}
 It follows from the mean value theorem that $p_{\epsilon}\,(v_1-g^{\epsilon}) - p_{\epsilon} (v_2-g^{\epsilon}) = p^{'}_{\epsilon} (y)
 (v_1-v_2)$ for some $y\in \R$, where $p^{'}_{\epsilon}(y)\geq 0$ thanks to \eqref{eq:prop-p-eps}-(iv).
 Now it follows from
 Lemma~\ref{lemma:maximum-principle}, with $c = r+ p_{\epsilon}'(y)$ that $v_1\geq v_2$ on $\R\times
 (0,T]$. The same argument applied to $v_2-v_1$ gives the reverse inequality.
\end{proof}

Utilizing the maximum principle, we will analyze properties of the sequence $(v^\epsilon)_{\epsilon\in (0,1)}$ in the following results.

\begin{lem}\label{lemma:v-eps-bounded}
Let Assumptions~\ref{ass: coeff}, \ref{ass: Levy measure} with $1\leq \alpha<2$, and \ref{ass: payoff} hold. Then for any $\epsilon\in (0,1)$,
\[
 0 \leq v^{\epsilon}\leq K+1 \quad \text{ on } E_T.
\]
\end{lem}

\begin{proof}
 It follows from Lemma~\ref{lemma:existence-penalty-prob} that
 $v^{\epsilon}$ is bounded on $E_T$ for each $\epsilon\in (0,1)$. In this proof, we will show that the bounds are uniform in $\epsilon$. First, it follows from \eqref{eq:prop-p-eps} part (i) that $(\partial_t - \L_{D} - I + r)\, v^{\epsilon} = - p_{\epsilon} (v^{\epsilon} - g^{\epsilon})\geq
 0$. Moreover, $v^{\epsilon} (x,0) = g^{\epsilon}(x)\geq 0$ for $x\in \R$. Then first inequality in the statement follows from Lemma~\ref{lemma:maximum-principle} directly. Second, consider $w= K+1- v^{\epsilon}$, it satisfies
 \begin{equation*}
  \left(\partial_t - \L_{D} - I + r\right) w = r (K+1) +
  p_{\epsilon} (v^{\epsilon} - g^{\epsilon}), \quad (x,t)\in \R\times (0,T].
 \end{equation*}
 Combining \eqref{eq:prop-p-eps} part (ii) and $g^\epsilon \leq K$, we have $p_{\epsilon}(K+1 - g^{\epsilon})=0$. Hence,
 \begin{equation}
  \left(\partial_t - \L_{D} - I + r\right) w + p_{\epsilon}
  (K+1-g^{\epsilon}) - p_{\epsilon}(v^{\epsilon}- g^{\epsilon}) =
  \left[\partial_t - \L_{D} - I + r + p_{\epsilon}^{'}(y)\right] w
  = r\,(K+1) \geq 0,
 \end{equation}
 where the first equality follows from the mean value theorem. Now applying Lemma~\ref{lemma:maximum-principle} to above equation with $c=r+p'(y)\geq 0$ (see \eqref{eq:prop-p-eps} part (iv)), we obtain $w(x,t) = K+1-v^\epsilon(x,t) \geq 0$ on $E_T$ for any $\epsilon\in (0,1)$, which confirms the second inequality in the statement of the lemma.
\end{proof}

\begin{lem}\label{lemma:v-x-eps-bounded}
Let Assumptions~\ref{ass: coeff}, \ref{ass: Levy measure} with $1\leq \alpha<2$, and \ref{ass: payoff} hold. Then for any $\epsilon\in (0,1)$,
\[
 \left|\partial_{x} v^{\epsilon} \right| \leq C \quad
 \text{ on } E_T,
\]
in which $C$ depends on $T$ and $L$.
\end{lem}
\begin{proof}
 Formally differentiating \eqref{eq:penalty-v-eps} with respect to $x$ gives the following equation:
 \begin{equation}\label{eq: pde v_x}
 \begin{split}
  & \left[\partial_t - a \partial^2_{xx} -(b+ \partial_x a) \partial_x - \hat{I} + \left(r-\partial_x b-\int_{|y|>1}\partial_x h\, \nu(dy)\right)\right]  w +v^\epsilon \partial_x r + p^{'}_{\epsilon}\left(v^\epsilon - g^{\epsilon}\right) \left(w- (g^{\epsilon})^{'}\right) =0,\\
  & \hspace{12cm} (x,t)\in \R\times (0,T],\\
  & w(x,0) = (g^{\epsilon})^{'}(x), \hspace{9.2cm} x\in \R.
 \end{split}
 \end{equation}
 Here $\hat{I} \phi := I \phi + \int_{\R} \left[\phi(x+h, t) - \phi(x,t)\right] \partial_x h \,\nu(dy)$, where the second integral is well defined for Lipschitz bounded function $\phi$ because $|\partial_x h|\leq |y|$ from \eqref{ass: h} and $\int_{|y|>1} |y|\nu(dy)<\infty$ from Assumption \ref{ass: Levy measure}.
 We will show that $\partial_x v^\epsilon$ is indeed a classical solution of \eqref{eq: pde v_x}. To this end, let us consider the equation
 \begin{equation*}
 \begin{split}
  & \left[\partial_t - a \partial^2_{xx} -(b+ \partial_x a) \partial_x - \hat{I} + \left(r-\partial_x b-\int_{|y|>1}\partial_x h\, \nu(dy)\right)\right]  w = -\partial_x r \,v^\epsilon - p^{'}_{\epsilon}\left(v^\epsilon - g^{\epsilon}\right) \left(\partial_x v^\epsilon- (g^{\epsilon})^{'}\right), \\
  & \hspace{12cm} (x,t)\in \R\times (0,T],\\
  & w(x,0) = (g^{\epsilon})^{'}(x), \hspace{9.2cm} x\in \R.
 \end{split}
 \end{equation*}
 Using Assumption~\ref{ass: coeff} and Lemma~\ref{lemma:existence-penalty-prob}, one can check that the driving term $-\partial_x r \,v^\epsilon - p^{'}_{\epsilon}\left(v^\epsilon - g^{\epsilon}\right) \left(\partial_x v^\epsilon- (g^{\epsilon})^{'}\right)$ and all coefficients of the previous equation are H\"{o}lder continuous. It then follows from Theorem~3.1 in \cite{garr-mena} on pp. 89 that the last equation has a classical solution, say $w$. Define $v(x,t):= \int_0^x w(z, t) dz + v^\epsilon(0,t)$. It is straight forward to check that $v$ is a classical solution of the following equation
 \[
 \begin{split}
  &(\partial_t -\L_D - I+r) v = -p_\epsilon(v^\epsilon - g^\epsilon), \quad (x,t)\in \R\times (0,T],\\
  & v(x,0) = g^\epsilon(x), \hspace{3.5cm} x\in \R.
 \end{split}
 \]
Since $g^\epsilon$ and $v^\epsilon$ are both bounded, then $-p_\epsilon(v^\epsilon-g^\epsilon)$ is also bounded. As a result, estimate (3.6) in Theorem~3.1 of \cite{garr-mena} on pp. 89 implies that $v$ is bounded solution of the last equation. However, Corollary~\ref{cor: uniqueness} already shows that $v^\epsilon$ is the unique bounded solution of the last solution, therefore $v=v^\epsilon$, hence $\partial_x v^\epsilon = w$ on $E_T$ and $\partial_x v^\epsilon$ is a classical solution of \eqref{eq: pde v_x}.

 Now we shall show $\partial_x v^\epsilon$ is bounded uniformly in $\epsilon$. Consider $\tilde{v} =e^{\gamma t}L + \partial_x v^\epsilon$, where $L$ is given by Assumption~\ref{ass: payoff} and $\gamma>0$ will be determined later. The function $\tilde{v}$ satisfies the following equation
 \begin{equation}\label{eq: est v_x}
 \begin{split}
  &\left[\partial_t - a\partial^2_{xx} - (b+\partial_x a) \partial_x - \hat{I} + r-\partial_x b -\int_{|y|>1}\partial_x h\, \nu(dy) + p_\epsilon^{'}(v^\epsilon-g^\epsilon)\right] \tilde{v}\\
  & \hspace{1cm} = \left(\gamma + r-\partial_x b -\int_{|y|>1}\partial_x h\, \nu(dy)\right) e^{\gamma t} L - \partial_x r\, v^\epsilon + p_\epsilon^{'}(v^\epsilon-g^\epsilon)\left(e^{\gamma t} L + (g^\epsilon)^{'}\right), \quad (x,t)\in \R\times (0,T],\\
  & \tilde{v}(x,0) = e^{\gamma t} L + (g^{\epsilon})^{'}(x), \hspace{9.8cm} x\in \R.
 \end{split}
 \end{equation}
 Recall that $\partial_x b$ and $\partial_x r$ are bounded from Assumption~\ref{ass: coeff}. Observe that $\int_{|y|>1} \partial_x h \,\nu(dy)<\infty$ because $|\partial_x h|\leq |y|$ and $\int_{|y|>1} |y| \,\nu(dy)<\infty$. Moreover, $v^\epsilon$ is bounded uniformly in $\epsilon$ thanks to Lemma~\ref{lemma:v-eps-bounded}. Therefore, one can find a sufficiently large $\gamma$, independent of $\epsilon$, such that $\left(\gamma + r-\partial_x b-\int_{|y|>1}\partial_x h\, \nu(dy)\right) e^{\gamma t} L - \partial_x r\, v^\epsilon>0$. On the other hand, $p_\epsilon^{'}(v^\epsilon-g^\epsilon)\left(e^{\gamma t} L + (g^\epsilon)^{'}\right)$ is also positive due to \eqref{eq:prop-p-eps}-(iv) and $|(g^{\epsilon})^{'}|\leq L$. As a result, the right-hand-side of \eqref{eq: est v_x} is positive. Now since $r- \partial_x b -\int_{|y|>1}\partial_x h\, \nu(dy) +p_\epsilon^{'}(v^\epsilon-g^\epsilon)$ is bounded from below, moreover the L\'{e}vy measure $(1+ \partial_x h) \nu(dy)$ associated to $\hat{I}$ satisfies $\int_{|y|>1} |y|(1+ |\partial_x h|) \nu(dy) \leq \int_{|y|>1} (|y|+|y|^2) \nu(dy)<\infty$ (see \eqref{ass: h} and Assumption \ref{ass: Levy measure}), we then have from Lemma~\ref{lemma:maximum-principle} with $I=\hat{I}$ that $\tilde{v}\geq 0$ on $E_T$. Hence $\partial_x v^\epsilon \geq - e^{\gamma T} L$ on $E_T$, for some positive $\gamma$ independent of $\epsilon$. The upper bound can be shown similarly by working with $\tilde{v} = e^{\gamma t}L - \partial_x v^\epsilon$.
\end{proof}

\begin{lem}\label{lemma:u-geq-g}
 Let Assumptions~\ref{ass: coeff}, \ref{ass: Levy measure} with $1\leq \alpha<2$, and \ref{ass: payoff} hold. Then for any $\epsilon \in (0,1)$,
 \[
  v^{\epsilon} \geq g^\epsilon \quad \text{ on } E_T.
 \]
\end{lem}
\begin{proof}
 Let us first show that $I g^{\epsilon}$ is uniformly bounded from below. Indeed,
 \begin{equation*}\label{eq:Ig-eps-lb}
 \begin{split}
  I g^{\epsilon} (x) &= \int_{|y|\leq 1} \nu(dy) \int_0^1 dz (1-z) \,h^2 \partial^2_{xx}
  g^{\epsilon}(x+z h) + \int_{|y|>1} \left[g^{\epsilon}(x+h) -
  g^{\epsilon}(x)\right] \nu(dy) \\
  & \geq -J \int_{|y|\leq 1} |y|^2 \nu(dy) - K \int_{|y|>1}
  \nu(dy),
 \end{split}
 \end{equation*}
 where the inequality follows from $\partial^2_{xx} g^\epsilon \geq -J$ and $0\leq g^\epsilon\leq K$.
 As a result, $\left(\partial_t - \L_{D} - I + r\right) g^{\epsilon}$ is bounded
 from above. This is because
 \begin{equation*}\label{eq:Lg-eps-ub}
 \begin{split}
  \left(\partial_t - \L_{D} - I + r\right) g^{\epsilon}(x) &= - a(x,t) \, \partial^2_{xx} g^{\epsilon}(x) - b(x,t) \, \partial_x g^{\epsilon}(x) + r(x,t)\, g^{\epsilon}(x) - I g^{\epsilon}(x)\\
  & \leq a^{(0)} J + |b|^{(0)} L + r^{(0)} K + J \int_{|y|\leq
  1} |y|^2 \nu(dy) + K \int_{|y|>1} \nu(dy)\\
  & = -p_{\epsilon}(0),
 \end{split}
 \end{equation*}
 where the second equality follows from \eqref{eq:prop-p-eps} part
 (iii). Therefore,
 \begin{equation*}
 \begin{split}
  \left(\partial_t - \L_{D} - I + r\right) \left(v^{\epsilon} -
  g^{\epsilon}\right) &= -p_{\epsilon} \left(v^{\epsilon} -
  g^{\epsilon}\right) - \left(\partial_t - \L_{D} - I + r\right)
  g^{\epsilon} \\
  & \geq -p_{\epsilon} \left(v^{\epsilon} - g^{\epsilon}\right) +
  p_{\epsilon}(0).
 \end{split}
 \end{equation*}
 The previous inequality and the mean value theorem combined imply that
 \begin{equation*}\label{eq:Lu-g-eps-lb-2}
  \left(\partial_t - \L_D - I + r + p_{\epsilon}^{'} (y)\right)
  \left(v^{\epsilon} - g^{\epsilon}\right) \geq 0,
 \end{equation*}
 for some $y\in \R$. Hence the statement of the lemma follows applying
 Lemma~\ref{lemma:maximum-principle} to the previous inequality
 and choosing $c=r+p_{\epsilon}^{'}(y)\geq 0$.
\end{proof}

\begin{cor}\label{cor:penalty-bounded}
 Let assumptions of Lemma~\ref{lemma:u-geq-g} hold. Then $p_{\epsilon}\left(v^{\epsilon} - g^{\epsilon}\right)$ is bounded
 uniformly in $\epsilon\in (0, 1)$.
\end{cor}

\begin{proof}
 Lemma~\ref{lemma:u-geq-g} and \eqref{eq:prop-p-eps}-(i)\&(iv) together imply that
 $p_{\epsilon}(0) \leq p_{\epsilon} \left(v^{\epsilon} - g^{\epsilon}\right) \leq 0.$
 Then the statement follows since $p_{\epsilon}(0)$ is independent of $\epsilon$; see \eqref{eq:prop-p-eps} part (iii).
\end{proof}

\subsection{Proof of Theorem~\ref{thm: main} and Corollary~\ref{cor: smooth-fit}}

\begin{proof}[Proof of Theorem~\ref{thm: main}]
 The proof consists of two steps. First, we show that there exists a function $v^*$ which solves \eqref{eq: VI} and $v^*\in W^{2,1}_p(B \times (s,T))$ for any integer $p\in (1,\infty)$, compact domain $B\subset \mathbb{R}$, and $s\in[0,T)$. Second, we confirm that $u^*(x,t) := v^*(x, T-t)$ is the value function for the problem \eqref{eq:optimal-stopping}.

 \textit{\underline{Step 1:}} First, it follows from Lemma~\ref{lemma:existence-penalty-prob} that $\partial_t v^\epsilon$, $\partial_x v^\epsilon$, and $\partial^2_{xx} v^\epsilon$ are continuous, hence locally bounded on $\R\times(0,T)$. Therefore $v^\epsilon\in W^{2,1}_{p, loc}(\R \times (0,T))$ for each $\epsilon \in (0,1)$.  Second, Lemmas~\ref{lemma:v-eps-bounded} and \ref{lemma:v-x-eps-bounded} show that $v^\epsilon$ and $\partial_x v^\epsilon$ are bounded on $E_T$, uniformly in $\epsilon$. Moreover, the penalty term $p_\epsilon(v^\epsilon -g^\epsilon)$ is also bounded uniformly in $\epsilon$ due to Corollary~\ref{cor:penalty-bounded}. Therefore these boundedness properties and Proposition~\ref{prop:W^21_p-est} with $f=-p_\epsilon(v^\epsilon - g^\epsilon)$ together imply that
 \begin{equation}\label{eq: uniform sobolev norm}
  \left\|v^\epsilon\right\|_{W^{2,1}_p(B\times (s,T))} \leq C, \quad \text{ for some constant } C \text{ independent of } \epsilon.
 \end{equation}
Thanks to the weak compactness of the Sobolev space $W^{2,1}_p$, $1<p<\infty$, we can then find a subsequence $(\epsilon_k)_{k\geq 0}$ converging to zero and a function $v^*$, such that $v^{\epsilon_{k}} \rightharpoonup v^*\in W^{2,1}_p (B\times (s,T))$. Here $``\rightharpoonup"$ represents the weak convergence; c.f. Appendix D.4. in \cite{evans} pp. 639. In fact this convergence can be shown to be pointwise and uniform in the index. Indeed, \eqref{eq: uniform sobolev norm} and the Sobolev embedding theorem (c.f. Lemma~3.3 in \cite{lad} pp. 80) combined imply that
 \[
  \left\|v^\epsilon\right\|^{(\beta)}_{B \times [s,T]} \leq C, \quad \text{ where } \beta = 2-\frac{3}{p} \text{ and } C \text{ is some constant independent of } \epsilon.
 \]
 Choosing $p>1$ so that $\beta>0$ and using the previous uniform estimate along with the Arzel\`{a}-Ascoli theorem, we then find a further subsequence of $(\epsilon_k)_{k\geq 0}$, which is still denoted by $(\epsilon_k)_{k\geq 0}$, such that $(v^{\epsilon_k})_{k\geq 0}$ converge to $v^*$ uniformly on $B\times [s,T]$. Since each $v^{\epsilon_k}$ is continuous, $v^*$ is also continuous on $B\times [s,T]$.

 Let us show that $v^*$ solves \eqref{eq: VI}. On the one hand, since $p_\epsilon(v^{\epsilon_k} - g^{\epsilon_k}) \leq 0$, we have $(\partial_t -\L_D - I+r) \,v^{\epsilon_k} \geq 0$ for each $\epsilon_k$. Hence
 \[
  \int \left(\partial_t - \mathcal{L}_D - I + r\right) v^* \phi \, dx dt =\lim_{\epsilon_k\rightarrow 0} \int \left(\partial_t - \mathcal{L}_D - I + r\right) v^{\epsilon_k} \phi \, dx dt \geq 0,
 \]
 for any compactly supported smooth function $\phi$. Here the identity above follows from applying the dual operator of $\partial_t - \mathcal{L}_D- I +r$ to $\phi$ and utilizing the dominated convergence theorem. The previous inequality then yields $(\partial_t -\L_D - I+r) \,v^* \geq 0$ on $B \times [s,T]$ in the distributional sense, which implies the same inequality on $\R\times (0,T]$ in the distributional sense, since the choices of $B$ and $s$ are arbitrary. On the other hand, Lemma~\ref{lemma:u-geq-g} shows that $v^{\epsilon_k} \geq g^{\epsilon_k}$. Then $v^*\geq g$ after sending $\epsilon_k \rightarrow 0$. Therefore, we obtain $\min\{(\partial_t -\L_D - I+r) \,v^*, v^*- g\}\geq 0$ on $\R\times (0,T]$ in the distributional sense. It then remains to show $(\partial_t -\L_D - I+r) \,v^* =0$ when $v^*>g$. To this end, take any $(x,t)$ such that $v^*(x,t)> g(x)$. Since both $v^*$ and $g$ are continuous, one can find a sufficiently small $\delta>0$ and a small neighborhood of $(x,t)$, such that $v^*(\tilde{x}, \tilde{t}) \geq g(\tilde{x}) + 2\delta$ for any $(\tilde{x}, \tilde{t})$ inside this neighborhood. Utilizing the uniform convergence of $(v^{\epsilon_k})_{k\geq 0}$ and $(g^{\epsilon_k})_{k\geq 0}$, we can then find sufficiently small $\epsilon_k$ such that $v^{\epsilon_k}(\tilde{x}, \tilde{t}) \geq g^{\epsilon_k}(\tilde{x}) + \delta$ in the aforementioned neighborhood. Hence $p_{\epsilon_k}(v^{\epsilon_k}-g^{\epsilon_k})(x,t) =0$, due to \eqref{eq:prop-p-eps}-(ii), which induces $(\partial_t -\L_D - I+r) \,v^{\epsilon_k}(x,t) =0$. After sending $\epsilon_k\rightarrow 0$, we conclude that $(\partial_t -\L_D - I+r) \,v^* =0$ in the distributional sense when $v^*>g$. Finally, since $v^*\in W^{2,1}_{p, loc}$, $v^*$ also solves \eqref{eq: VI} at almost every point in $E_T$.

 \textit{\underline{Step 2:}} Let us first show that $v^*$ is a viscosity solution of \eqref{eq: VI}. We will use the definition of viscosity solutions in \cite{pham-control}. Denote by $C_1(E_T)$ the class of functions which have at most linear growth, i.e., $|\phi(x,t)|\leq C(1+|x|)$ for some $C$ and any $(x,t)\in E_T$. Then viscosity solutions of \eqref{eq: VI} are defined as follows:
 Any $v\in C(E_T)$ is a viscosity supersolution (subsolution) of \eqref{eq: VI} if
 \[
 \begin{split}
  & \min\{\partial_t \phi - \L_D \phi - I \phi + r v, v - g\} \geq 0 \,(\leq 0), \quad (x,t)\in \R\times (0,T],\\
  & v(x,0) \geq g(x) \,(\leq g(x)), \hspace{3.7cm} x\in \R,
 \end{split}
 \]
 for any function $\phi\in C^{2,1}(\R\times (0,T)) \cap C_1(E_T)$ such that $v(x,t)=\phi(x,t)$ and $v(\tilde{x}, \tilde{t}) \geq \phi(\tilde{x}, \tilde{t})$ ($v(\tilde{x}, \tilde{t}) \leq \phi(\tilde{x}, \tilde{t})$) for any other point $(\tilde{x},\tilde{t})\in \R\times (0.T)$. The function $v$ is said to be a viscosity solution of \eqref{eq: VI} if it is both supersolution and subsolution.

 Let us show that $v^*$ is a viscosity subsolution of \eqref{eq: VI}. Fix $(x,t)\in \R\times (0,T]$, consider $v^*(x,t)>g(x)$, otherwise $\min\{\partial_t \phi - \L_D \phi - I \phi + r v^*, v^*(x,t) - g(x)\} \leq 0$ is automatically satisfied. Without loss of generality we can assume that $(x,t)$ is the strict maximum of $v^* -\phi$ in a neighborhood $B(x,t; \delta)$, otherwise the test function can be modified appropriately. On the other hand, since $(v^{\epsilon_k})_{k\geq 0}$ converges to $v^*$ uniformly in compact domains, we can find sufficiently small $\epsilon_k$ such that $v^{\epsilon_k} -\phi$ attains its maximum over $B(x,t; \delta)$ at $(x_k, t_k)\in B(x,t; \delta)$. Moreover, $(x_k, t_k) \rightarrow (x,t)$ as $\epsilon_k\rightarrow 0$. Since $v^{\epsilon_k}$ is a classical solution of \eqref{eq:penalty-v-eps} (see Lemma~\ref{lemma:existence-penalty-prob}), it is also a viscosity solution. Hence $(\partial_t - \L_D - I + r)\phi(x_k, t_k) + p_{\epsilon_k}(v^{\epsilon_k}(x_k, t_k) - g^\epsilon(x_k)) \leq 0$. Now, since $v^*(x,t)> g(x)$ and $v^{\epsilon_k}(x_k, t_k) - g(x_k)$ converges to $v^*(x,t) - g(x)$, we obtain $\lim_{\epsilon_k \rightarrow 0} p_{\epsilon_k}(v^{\epsilon_k}(x_k, t_k) - g^\epsilon(x_k)) =0$. As a result, $(\partial_t - \L_D - I + r)\phi(x, t)\leq 0$ by sending $\epsilon_k\to 0$. This confirms that $v^*$ is a viscosity subsolution of \eqref{eq: VI}.

 For the supersolution property, since $v^* \geq g$, it suffices to show that $(\partial_t - \L_D - I + r)\phi(x, t)\geq 0$ for any test function $\phi$. This follows from the similar arguments which we used for the subsolution property in the previous paragraph.

 Define $u^*(x,t) = v^*(x, T-t)$. It is clear that $u^*$ is a viscosity solution of \eqref{eq:varineq-u}. Then the statement follows from Theorem~4.1 in \cite{pham-control}, which states that the value function $u$ is the unique viscosity solution of \eqref{eq:varineq-u} when the L\'{e}vy measure satisfies $\int_{|y|>1} |y|^2 \nu(dy)<\infty$.
\end{proof}

\begin{proof}[Proof of Corollary~\ref{cor: smooth-fit}]
 (i) Combining Theorem~\ref{thm: main} and the Sobolev embedding theorem (c.f. Lemma~3.3 in \cite{lad} pp. 80), we have $u\in H^{\beta, \frac{\beta}{2}}(D\times [0, T-s])$, where $\beta = 2- \frac{3}{p}$ and $s<T$. Choosing $p>3$ so that $\beta>1$, the continuity of $\partial_x u$ follows from Definition~\ref{def:sobolev_holder space}.

 (ii) Let us first show that $Iu$ is well defined and H\"{o}lder continuous. Since $u\in H^{\beta, \frac{\beta}{2}}(D\times [0, T-s])$ (which follows due to (i)), choosing sufficiently large $p$ so that $\beta>\alpha$, $Iu \in H^{\beta-\alpha-\gamma, \frac{\beta-\alpha-\gamma}{2}}(\overline{D_{T-s}})$ by Lemma~\ref{lemma:infinite-var-holder} part (i). Now, for $B\subset \R$ and $t_1, t_2\in [0,T)$ such that $B\times (t_1, t_2) \subset \C$, consider the following boundary value problem:
 \begin{equation}\label{eq: bvp}
 \begin{split}
  & (-\partial_t - \L_D + r) \,v = I u, \quad (x,t) \in B\times[t_1, t_2),\\
  & v(x,t) = u(x,t), \hspace{1.5cm} (x,t)\in \partial B \times [t_1, t_2) \cup \overline{B} \times t_2.
 \end{split}
 \end{equation}
 It is straightforward to show that $u$ is the unique viscosity solution for the previous problem using the fact that $u$ is the unique viscosity solution for \eqref{eq:varineq-u}. On the other hand, since the boundary and terminal values of \eqref{eq: bvp} are continuous and the driving term $Iu$ is H\"{o}lder continuous, it follows from Theorem 9 in \cite{friedman} pp. 69 that \eqref{eq: bvp} has a classical solution $u^*\in C^{2,1}(B\times (t_1, t_2))$. Hence $u = u^*$ on $B\times (t_1, t_2)$, since $u^*$ is also a viscosity solution. Therefore, $u\in C^{2,1}(B\times (t_1, t_2))$. The statement now follows, since $B\times (t_1, t_2)$ is an arbitrary subset of $\C$.
\end{proof}

\bibliographystyle{plain}

\end{document}